\documentclass[11pt,reqno,a4paper]{amsart}
\usepackage{amssymb,amscd,amsmath}
\usepackage{pspicture}
\usepackage{graphicx}
\usepackage{mathptmx}

\usepackage[matrix,arrow,curve,frame]{xy}    

\xymatrixcolsep{1.9pc}                          
\xymatrixrowsep{1.9pc}
\newdir{ >}{{}*!/-5pt/\dir{>}}                  

 \calclayout

\makeatletter

\renewcommand{\subsection}{\@startsection{subsection}{1}{0pt}{-3.25ex plus -1ex minus-.2ex}{1.5ex plus.2ex}{\normalfont\it}}
\renewcommand{\section}{\@startsection{section}{1}{\parindent}{3.5ex plus 1ex minus .2ex}{2.3ex plus.2ex}{\sc}}

\renewcommand{\phi}{\varphi}

\renewcommand{\geq}{\geqslant}
\renewcommand{\epsilon}{\varepsilon}

\renewcommand{\kappa}{\varkappa}
 
\DeclareMathOperator{\spec}{Spec}
  \DeclareMathOperator{\Cyl}{Cyl}

\DeclareMathOperator{\hocolim}{hocolim}
\DeclareMathOperator{\chr}{char} 
 
 \DeclareMathOperator{\mot}{mot}

\DeclareMathOperator{\tw}{tw} 
\DeclareMathOperator{\Hom}{Hom} 
 
 \DeclareMathOperator{\id}{id}
\DeclareMathOperator{\corr}{Cor} 
\DeclareMathOperator{\Mor}{Mor} 
\DeclareMathOperator{\colim}{colim} \DeclareMathOperator{\Ho}{Ho}

 \DeclareMathOperator{\Ab}{Ab}
 
\DeclareMathOperator{\Ch}{Ch} 
 
\DeclareMathOperator{\coker}{Coker} \DeclareMathOperator{\nis}{nis}

 \DeclareMathOperator{\Mod}{Mod}

\newcommand{\lra}[1]{\bl{#1}\longrightarrow\relax}
\newcommand{\bl}[1]{\buildrel #1\over}
\newcommand{\cc}{\mathcal}
\newcommand{\bb}{\mathbb}

\newcommand{\op}{{\textrm{\rm op}}}

\newcommand{\wt}{\widetilde}

\newcommand{\corrt}{\wt{\corr}}

\newtheorem{thm}{Theorem}[section]
\newtheorem{prop}[thm]{Proposition}

\newtheorem{cor}[thm]{Corollary}
\newtheorem{lem}[thm]{Lemma}

\newtheorem{rem}[thm]{Remark}
\newtheorem*{thmm}{Theorem}

\newtheorem{defs}[thm]{Definition}

\makeatother

\begin{document}

\footskip30pt


\title{Reconstructing rational stable motivic homotopy theory}
\author{Grigory Garkusha}
\address{Department of Mathematics, Swansea University, Fabian Way, Swansea SA1 8EN, UK}
\email{g.garkusha@swansea.ac.uk}

\urladdr{http://math.swansea.ac.uk/staff/gg/}

\keywords{Motivic homotopy theory, generalized correspondences, triangulated categories of motives}

\subjclass[2010]{14F42; 14F05}

\begin{abstract}
Using a recent computation of the rational minus part of $SH(k)$ by
Ananyevskiy--Levine--Panin~\cite{ALP}, a theorem of
Cisinski--D\'eglise~\cite{CD} and a version of the
R\"ondigs--{\O}stv{\ae}r~\cite{RO} theorem, rational stable motivic
homotopy theory over an infinite perfect field of characteristic
different from 2 is recovered in this paper from finite Milnor--Witt
correspondences in the sense of Calm\`es--Fasel~\cite{CF}.
\end{abstract}

\maketitle

\thispagestyle{empty} \pagestyle{plain}

\newdir{ >}{{}*!/-6pt/@{>}} 


\section{Introduction}

By the celebrated Serre finiteness theorem~\cite{Serre} the
positive stable homotopy groups of the classical sphere spectrum
with rational
coefficients 
are zero. It implies that the stable homotopy category of
$S^1$-spectra with rational coefficients $SH_{\bb Q}$ is naturally
equivalent to the homotopy category of $H\bb Q$-modules $\Ho(\Mod
H\bb Q)$, where $H\bb Q$ is the Eilenberg--Mac~Lane symmetric
spectrum of $\bb Q$. By the Robinson theorem~\cite{Rob} the homotopy
category of $HA$-modules $\Ho(\Mod HA)$, where $A$ is a ring with
identity, is equivalent to the derived category $D(A)$ of $A$. Thus
$SH_{\bb Q}$ is naturally equivalent to the derived category of
rational vector spaces $D(\bb Q)$.



In the motivic world the role of a ring is played by a ``preadditive
category of correspondences" $\cc A$ whose objects are the smooth
algebraic varieties $Sm/k$ over a field $k$. Using the category
theory terminology, $\cc A$ is a ring with several objects, whose
objects are those of $Sm/k$. In turn, the role of the classical
derived category over a ring is played by the category $DM_{\cc
A}(k)$, which is just an extension of the celebrated Voevodsky
triangulated category~\cite{Voe1} $DM(k)$ to general
correspondences. Since motivic homotopy theory requires the
Nisnevich topology and contractibility of the affine line $\bb A^1$,
we require the relevant properties for $\cc A$ to satisfy (see
Section~\ref{sectcorrs} for details).

The rational stable motivic homotopy theory $SH(k)_{\bb Q}$ splits
in two parts: $SH^+(k)_{\bb Q}$ and $SH^-(k)_{\bb Q}$. The plus part
$SH^+(k)_{\bb Q}$ is equivalent to Voevodsky's $DM(k)_{\bb Q}$ (this
follows from a theorem of Cisinski--D\'eglise~\cite[16.2.13]{CD}).
Ananyevskiy--Levine--Panin~\cite{ALP} have computed $SH^-(k)_{\bb
Q}$ as the category of Witt motives with rational coefficients (see
Bachmann~\cite{Bach} as well). Using these results, we show in
Theorem~\ref{ratsphcw} that the rational motivic sphere spectrum
$\bb S\otimes\bb Q$ is naturally equivalent to the ``additive
motivic sphere spectrum" $\bb S^{MW}\otimes\bb Q$ associated with
the additive category of finite Milnor--Witt correspondences
$\corrt$ in the sense of Calm\`es--Fasel~\cite{CF}.

Next, we extend the R\"ondigs--{\O}stv{\ae}r
theorem~\cite[Theorem~1]{RO} to the triangulated category $DM_{\cc
A}(k)$ (see Theorem~\ref{rondigsostvaer}). {\it This extension is of
independent interest!} For example, it is of great utility to
compare various triangulated categories of motives in~\cite{GG}.
The generalised R\"ondigs--{\O}stv{\ae}r theorem can also be
regarded as a motivic counterpart of the Robinson
theorem~\cite{Rob}.


Theorem~\ref{ratsphcw} computing $\bb S\otimes\bb Q$ together with
generalised R\"ondigs--{\O}stv{\ae}r's Theorem~\ref{rondigsostvaer}
lead to the proof of the main result of the paper which is
formulated as follows (see Theorem~\ref{reconstrshq}).

\begin{thmm}[Reconstruction]
If $k$ is an infinite perfect field of characteristic not 2, then
$SH(k)_{\bb Q}$ is equivalent to the triangulated category of
Milnor--Witt motives with rational coefficients $DM_{MW}(k)_{\bb Q}$
in the sense of~\cite{DF}. The equivalence preserves the
triangulated structure.
\end{thmm}

One of the approaches to constructing motivic homotopy theory,
pioneered by Voevodsky, is to use various correspondences on smooth
algebraic varieties. This approach has many computational
advantages. Voevodsky constructed~\cite{Voe1} the category of
motives $DM(k)$ by using finite correspondences. Later he developed
the theory of framed correspondences~\cite{Voe2}. One of the aims
was to suggest another framework for Morel--Voevodsky's stable
motivic homotopy theory $SH(k)$. In~\cite{GP3} the author and Panin
use Voevodsky's theory to develop the theory of big framed motives
which converts the classical Morel--Voevodsky stable motivic
homotopy theory into an equivalent local theory of framed bispectra.

One of the central objects of the theory of (big) framed motives in
the sense of~\cite{GP3} is linear framed motives of algebraic
varieties. They are explicitly constructed complexes of Nisnevich
sheaves with framed correspondences $\bb
ZF(-\times\Delta^\bullet,X)$, where $X\in Sm/k$. As an application
of the Reconstruction Theorem we prove the following result
comparing motivic complexes with framed and Milnor--Witt
correspondences respectively (see Theorem~\ref{complexesq}).

\begin{thmm}[Comparison]
Given an infinite perfect field of characteristic not 2 and a
$k$-smooth scheme $X$, each morphism of complexes of Nisnevich
sheaves
   $$f_n:\bb ZF(-\times\Delta^\bullet,X\times\bb G_m^{\wedge n})\otimes\bb Q\to
        \corrt(-\times\Delta^\bullet,X\times\bb G_m^{\wedge n})_{\nis}\otimes\bb Q,\quad n\geq 0,$$
is a quasi-isomorphism, where the left complex is the $n$-twisted
linear framed motive of $X$ with rational coefficients in the sense
of~\cite{GP3}.
\end{thmm}

The author would like to thank I.~Panin for numerous discussions on
motivic homotopy theory. He is also grateful to A.~Druzhinin,
J.~Fasel and A.~Neshitov for helpful discussions on Milnor--Witt
correspondences. The author thanks D.-C. Cisinski for pointing out
results of Riou, thanks to which the main theorem of the paper has
been improved. This paper was written during the visit of the author
to IHES in September 2016. He would like to thank the Institute for
the kind hospitality.

Throughout the paper we denote by $Sm/k$ the category of smooth
separated schemes of finite type over the base field $k$.

\section{Additive categories of correspondences}\label{sectcorrs}

In this section we set up a framework within which we shall work later.

\begin{defs}\label{tak}{\rm
We say that a preadditive category $\cc A$ is a {\it category of
correspondences\/} if:

\begin{enumerate}
\item Its objects are those of $Sm/k$. Its morphisms are also referred to as {\it $\cc A$-correspondences}
or just {\it correspondences}.


\item There is a functor $\rho:Sm/k\to\cc A$, which is the identity map on objects. The image $\rho(f)$
of a morphism of smooth schemes $f:X\to Y$ will be referred to as the {\it graph of $f$\/}
and denoted by $\Gamma_f$. We have in particular that $\Gamma_{gf}=\Gamma_g\circ\Gamma_f$ and $\Gamma_{\id}=\id$.
Thus we have a functor
   $$\cc A:(Sm/k)^{\op}\times Sm/k\to\Ab,\quad (X,Y)\mapsto\cc A(X,Y),$$
such that $\cc A(1_X,g)=\Gamma_g\circ-$ and $\cc A(h,1_Y)=-\circ\Gamma_h$.

\item For every elementary Nisnevich square
   $$\xymatrix{U'\ar[r]\ar[d]&X'\ar[d]\\
               U\ar[r]&X}$$
the sequence of Nisnevich sheaves
   $$0\to\cc A(-,U')_{\nis}\to\cc A(-,U)_{\nis}\oplus\cc A(-,X')_{\nis}\to\cc A(-,X)_{\nis}\to 0$$
is exact. Moreover, we require $\cc A(-,\emptyset)_{\nis}=0$ (corresponding to the ``degenerate distinguished
square", $\emptyset$, with only one entry in the lower right-hand corner).

\item For every $\cc A$-presheaf $\cc F$ (i.e. an additive contravariant
functor from $\cc A$ to Abelian groups Ab)
the associated Nisnevich sheaf $\cc F_{\nis}$ has a unique structure of an
$\cc A$-presheaf for which the canonical morphism $\cc F\to\cc
F_{\nis}$ is a morphism of $\cc A$-pre\-sheaves.

\item There is an action of $Sm/k$ on $\cc A$
in the following sense. Given $U\in Sm/k$ there is a homomorphism
   $$\alpha_U:\cc A(X,Y)\to\cc A(X\times U,Y\times U),$$
functorial in $X$ and $Y$, such that for any morphism $f:U\to V$ in
$Sm/k$ the following square of Abelian groups is commutative
   $$\xymatrix{\cc A(X\times V,Y\times V)\ar[rrr]^{\cc A(1_X\times f,1_{Y\times V})}&&&\cc A(X\times U,Y\times V)\\
               \cc A(X,Y)\ar[u]_{\alpha_V}\ar[rrr]^{\alpha_U}&&&\cc A(X\times U,Y\times U).
               \ar[u]_{\cc A(1_{X\times U},1_Y\times f)}}$$
We require $\alpha_U(\id_X)=\id_{X\times U}$ for all $U,Z\in Sm/k$.
By the functoriality of $\alpha_U$ in $X$ we mean that the following
square of Abelian groups is commutative for any $Y\in
Sm/k$ and any morphism $f: X^{\prime} \to X$ in $\cc A$
$$\xymatrix{\cc A(X\times U,Y\times U)\ar[rrr]^{\cc A(f\times 1_U,1_{Y\times U})}&&&\cc A(X^{\prime}\times U,Y\times U)\\
               \cc A(X,Y)\ar[u]_{\alpha_U}\ar[rrr]^{\cc A(f,1_Y)}&&&\cc A(X^{\prime},Y).
               \ar[u]_{\alpha_U}}$$
By the functoriality of $\alpha_U$ in $Y$ we mean that the following
square of additive functors is commutative for any $X \in
Sm/k$ and any morphism $g: Y \to Y^{\prime}$ in $\cc A$
$$\xymatrix{\cc A(X\times U,Y\times U)\ar[rrr]^{\cc A(1_{X\times U},g\times 1_U)}&&&\cc A(X\times U,Y^{\prime}\times U)\\
               \cc A(X,Y)\ar[u]_{\alpha_U}\ar[rrr]^{\cc A(1_X,g)}&&&\cc A(X,Y^{\prime}).
               \ar[u]_{\alpha_U}}$$
\end{enumerate}

In other words, we have a functor
      $$\boxtimes:\cc A\times Sm/k\to\cc A$$
sending $(X,U)\in Sm/k\times Sm/k$ to $X\times U\in Sm/k$ and
such that $1_X\boxtimes f=\Gamma_{1_X\times f}$,
$(u+v)\boxtimes f=u\boxtimes f+v\boxtimes f$ for all
$f\in\Mor(Sm/k)$ and $u,v\in\Mor(\cc A)$.
}\end{defs}

\begin{rem}{\rm
It follows from Definition~\ref{tak}(3) that the canonical morphism
   $$\cc A(-,X)_{\nis}\oplus\cc A(-,Y)_{\nis}\to\cc A(-,X\sqcup Y)_{\nis}$$
is an isomorphism of Nisnevich sheaves.

}\end{rem}

Observe that for any category of correspondences $\cc A$, an $\cc A$-presheaf $\cc F$ and $U\in
Sm/k$ the presheaf
   $$\underline\Hom(U,\cc F):=\cc F(-\times U)$$
is an $\cc A$-presheaf. Moreover, it is functorial in $U$.

For instance $\cc A$ can be given by the naive preadditive category
of correspondences $\cc A_{naive}$ with $\cc A_{naive}(X,Y)$ being
the free abelian group generated by $\Hom_{Sm/k}(X,Y)$. Non-trivial
examples are given by finite correspondences $\corr$ in the sense of
Voevodsky~\cite{Voe1}, finite Milnor--Witt correspondences $\corrt$
in the sense of Calm\`es--Fasel~\cite{CF} or $K_0^{\oplus}$ in the
sense of Walker~\cite{Wlk}. Given a ring $R$ (not necessarily
commutative) which is flat as a $\bb Z$-algebra and a category of
correspondences $\cc A$, we can form an {\it additive category of
correspondences\/} $\cc A_R$ with coefficients in $R$. By
definition, $\cc A_R(X,Y):=\cc A(X,Y)\otimes R$ for all $X,Y\in
Sm/k$.

\begin{defs}\label{gret}{\rm
We say that a category of correspondences $\cc A$ is a {\it
$V$-category of correspondences\/} (``$V$" for Voevodsky) if for any
$\bb A^1$-invariant $\cc A$-presheaf of abelian groups $\cc F$ the
associated Nisnevich sheaf $\cc F_{\nis}$ is $\bb A^1$-invariant.
Recall that a Nisnevich sheaf $\cc F$ of abelian groups is {\it
strictly $\bb A^1$-invariant\/} if for any $X\in {Sm}/k$, the
canonical morphism
   $$H^*_{\nis}(X,\cc F)\to H^*_{\nis}(X\times\bb A^1,\cc F)$$
is an isomorphism.
A $V$-category of correspondences $\cc A$ is a {\it strict
$V$-category of correspondences\/} if for any $\bb A^1$-invariant
$\cc A$-presheaf of abelian groups $\cc F$ the associated Nisnevich
sheaf $\cc F_{\nis}$ is strictly $\bb A^1$-invariant.

}\end{defs}

Observe that any (strict) $V$-category of correspondences is a
(strict) $V$-ringoid in the sense of~\cite{GP1}. For example $\corr$
and $K^\oplus_0$ are $V$-categories of correspondences, which are
strict whenever the base field $k$ is perfect (see~\cite{Voe}
and~\cite{Wlk}). The category $\corrt$ is a $V$-category of
correspondences, which is strict if $k$ is infinite and perfect with $\chr k\not=2$
(see~\cite{DF,Kol}). Observe that if $\cc A$ is a $V$-category of
correspondences then so is $\cc A_R$ with $R$ commutative flat as a
$\bb Z$-algebra. Moreover, if $R$ is a ring of fractions of $\bb Z$ like,
for example, $\bb Z[\frac1p]$ or $\bb Q$, then $\cc A_R$ is a strict
$V$-category of correspondences whenever $\cc A$ is.

Let $\cc A$ be a category of correspondences. Let $Sh(Sm/k)$
(respectively $Sh(\cc A)$) denote the category of Nisnevich sheaves
on $Sm/k$ (respectively Nisnevich $\cc A$-sheaves). Similar
to~\cite[6.4]{GP} $Sh(\cc A)$ is a Grothendieck category such that
$\{\cc A(-,X)_{\nis}\}_{X\in Sm/k}$ is a family of generators of
$Sh(\cc A)$. Denote by $D(Sh(Sm/k))$ and $D(Sh(\cc A))$ the
corresponding derived categories of unbounded complexes. Note that
$D(Sh(Sm/k))=D(Sh(\cc A_{naive}))$.

The category $\bb M$ of motivic spaces consists of contravariant
functors from $Sm/k$ to pointed simplicial sets. We refer the reader
to~\cite{Jar2,MV} for the definition of motivic weak equivalences
between motivic spaces.

\begin{lem}\label{datsyuk}
Given any field $k$, let $\cc A$ be a category of correspondences.
Then the natural map
   $$f:\cc A(-,X\times\bb A^1)\to\cc A(-,X)$$
is a motivic weak equivalence in the category of motivic spaces $\bb M$.
\end{lem}

\begin{proof}
We follow an argument of~\cite[p.~694]{RO}. As in classical
algebraic topology, an inclusion of pointed motivic spaces $g:A\to B$ is an
$\bb A^1$-deformation retract if there exist a map $r:B\to A$ such
that $rg=\id_A$ and an $\bb A^1$-homotopy $H:B\wedge\bb A^1_+\to B$
between $gr$ and $\id_B$ which is constant on $A$. Then $\bb
A^1$-deformation retracts are motivic weak equivalences.

There is an obvious map $r:\cc A(-,X)\to\cc A(-,X\times\bb A^1)$
such that $fr=1$. Since $Sm/k$ naturally acts on $\cc A$, it follows that
$\cc A(-\times\bb A^1,X\times\bb A^1)$ is an $\cc A$-presheaf.

There is a natural isomorphism
   $$\Hom(\cc A(-,X\times\bb A^1),\cc A(-\times\bb A^1,X\times\bb A^1))
     \cong\cc A(X\times\bb A^1\times\bb A^1,X\times\bb A^1),$$
where the Hom-set on the left is taken in the category of $\cc A$-presheaves.
Consider the functor $\rho:Sm/k\to\cc A$.
Denote by $\alpha$ the obvious map $\bb A^1\times\bb A^1\to\bb
A^1$. We set $h=\rho(1_X\times\alpha)$; then $h$ uniquely determines a
morphism of $\cc A$-presheaves
   $$h':\cc A(-,X\times\bb A^1)\to\cc A(-\times\bb A^1,X\times\bb A^1).$$
This morphism can be regarded as a morphism in $\bb M$,
denoted by the same letter. By adjointness $h'$
uniquely determines a map in $\bb M$
   \begin{gather*}
    H:\cc A(-,X\times\bb A^1)\wedge \bb A^1_+\to\cc A(-,X\times\bb A^1).
   \end{gather*}
Then $H$ yields an $\bb A^1$-homotopy between the identity map
and $rf$. We see that $f$ is a motivic weak equivalence, as required.
\end{proof}

By the general localization theory of compactly generated
triangulated categories~\cite{N} one can localize
$D(Sh(\cc A))$ with respect to the localizing subcategory $\cc L$
generated by complexes of the form
   $$\cdots\to 0\to\cc A(-,X\times\bb A^1)_{\nis}\xrightarrow{pr_X}\cc A(-,X)_{\nis}\to 0\to\cdots,\quad X\in Sm/k.$$
The resulting quotient category $D(Sh(\cc A))/\cc L$ is denoted by
$D_{\bb A^1}(Sh(\cc A))$.

If we denote by $DM_{\cc A}^{eff}(k)$ the full subcategory of $D(Sh(\cc A))$ consisting of
the complexes with strictly $\bb A^1$-invariant homology sheaves, then similar to a theorem of Voevodsky~\cite{Voe}
the composite functor
   $$DM_{\cc A}^{eff}(k)\hookrightarrow D(Sh(\cc A))\to D_{\bb A^1}(Sh(\cc A))$$
is an equivalence of triangulated categories whenever $\cc A$ is a strict $V$-category of correspondences. Moreover, the functor
   $$C_*:D(Sh(\cc A))\to D(Sh(\cc A)),\quad X\mapsto Tot(X(-\times\Delta^\bullet)),$$
lands in $DM_{\cc A}^{eff}(k)$. The kernel of $C_*$ is $\cc L$ and $C_*$ is left adjoint to the inclusion
functor
   $$i:DM_{\cc A}^{eff}(k)\hookrightarrow D(Sh(\cc A))$$
(see~\cite{Voe} for details or~\cite[3.5]{GP1}).

Let $(\bb G_m,1)\in\bb M$ denote $\bb G_m$ pointed at 1 and let $\bb G_m^{\cc A}$ be the sheaf
   $$\coker(\cc A(-,pt)_{\nis}\to\cc A(-,\bb G_m)_{\nis}),$$
induced by the map $pt\mapsto 1\in\bb G_m$ in $Sm/k$. Regarding it as a complex concentrated in zeroth degree,
we have an endofunctor
   $$-\boxtimes\bb G_m^{\cc A}:D_{\bb A^1}(Sh(\cc A))\to D_{\bb A^1}(Sh(\cc A)),$$
induced by the action of $Sm/k$ on $\cc A$. In more detail,
by~\cite[3.4]{AG} $\Ch(Sh(\cc A))$ is a Grothendieck category with
generators of the form $\{D^n\cc A(-,U)_{\nis}\}_{n\in\bb Z,U\in
Sm/k}$. Here $D^n\cc A(-,U)_{\nis}$ is the complex which is $\cc
A(-,U)_{\nis}$ in degrees $n$ and $n-1$ and 0 elsewhere, with
interesting differential being the identity map. Every complex
$X\in\Ch(Sh(\cc A))$ is written as a colimit of generators
   $$X=\colim_{(D^n\cc A(-,U)_{\nis}\to X)}D^n\cc A(-,U)_{\nis}.$$
We set,
   $$X\boxtimes\bb G_m^{\cc A}:=\colim_{(D^n\cc A(-,U)_{\nis}\to X)}D^n\cc A(-,U\wedge\bb G_m^{\wedge 1})_{\nis},$$
where the sheaf $\cc A(-,U\wedge\bb G_m^{\wedge 1})_{\nis}:=\coker(\cc A(-,U\times pt)_{\nis}\to\cc A(-,U\times\bb G_m)_{\nis})$.

Stabilizing $D_{\bb A^1}(Sh(\cc A))$ in the $\bb G_m$-direction with
respect to this endofunctor, we arrive at the category $D_{\bb
A^1}^{st}(Sh(\cc A))$. If $\cc A$ is a strict $V$-category of
correspondences, we can likewise stabilize $DM_{\cc A}^{eff}(k)$ in
the $\bb G_m$-direction. The resulting category is denoted by
$DM_{\cc A}(k)$. The triangulated equivalence $C_*:D_{\bb
A^1}(Sh(\cc A))\to DM_{\cc A}^{eff}(k)$ extends to a triangulated
equivalence
   $$C_*:D_{\bb A^1}^{st}(Sh(\cc A))\to DM_{\cc A}(k).$$

Given a category of correspondences $\cc A$ and $p>0$, we shall write $D^{st}_{\bb A^1}(Sh(\cc A))[p^{-1}]$
(respectively $D^{st}_{\bb A^1}(Sh(\cc A))_{\bb Q}$) to denote the
category $D^{st}_{\bb A^1}(Sh(\cc A\otimes\bb Z\left[\frac1p\right]))(k)$
(respectively $D^{st}_{\bb A^1}(Sh(\cc A\otimes\bb Q))(k)$). Note that $\cc
A\otimes\bb Z\left[\frac1p\right]$ and $\cc A\otimes\bb Q$ are
categories of correspondences.

\begin{defs}{\rm
We say that a category of correspondences $\cc A$ is {\it
symmetric monoidal\/} if the usual product of schemes defines a
symmetric monoidal structure on $\cc A$.

}\end{defs}

The categories $\corr$, $\corrt$, $K_0^{\oplus}$ are examples of
symmetric monoidal $V$-categories (see~\cite{CF,DF,Sus,SV1,Wlk} for
more details). $\cc A_{naive}$ is obviously symmetric monoidal.

Given a symmetric monoidal category of correspondences $\cc A$,
a theorem of Day~\cite{Day} implies that the category of $\cc A$-presheaves
$PSh(\cc A)$ is a closed symmetric monoidal category with a tensor
product defined as
   $$X\otimes Y=\int^{(U,V)\in\cc A\times\cc A}X(U)\otimes Y(V)\otimes\cc A(-,U\times V).$$
The monoidal unit equals $\cc A(-,pt)$ with $pt=\spec k$.

The tensor product is then extended to a tensor product $\wt\otimes$
on $Sh(\cc A)$. Namely, for all $F,G\in Sh(\cc A)$ we set
$F\wt{\otimes}G$ to be the sheaf associated with the presheaf
$F\otimes G$ defined above. With this tensor product $Sh(\cc A)$ is
a closed symmetric monoidal category with $\cc A(-,pt)_{\nis}$ a
monoidal unit. Likewise, $\wt\otimes$ is extended to chain complexes
$\Ch(Sh(\cc A))$ which also defines a closed symmetric monoidal
structure on the derived category $D(Sh(\cc A))$ with respect to the
derived tensor product $\wt{\otimes}^L$ (we also refer the reader to~\cite[Section~2]{SV1}
and~\cite[3.3]{CD1}). It is straightforward to
show that the localizing subcategory $\cc L$ of $D(Sh(\cc A))$
defined above is closed under the derived tensor product
$\wt{\otimes}^L$. As a result, one obtains a symmetric monoidal
product on $D^{st}_{\bb A^1}(Sh(\cc A))$ (and on $DM^{eff}_{\cc A}(k)$,
$DM_{\cc A}(k)$ if $\cc A$ is a strict $V$-category).

\begin{rem}\label{dmsymm}{\rm
Let $\cc A$ be a symmetric monoidal strict $V$-category of correspondences.
With a little extra care we describe the tensor product in $DM_{\cc A}(k)$ explicitly as follows.
The endofuctor $-\boxtimes\bb G_m^{\cc A}:\Ch(Sh(\cc A))\to\Ch(Sh(\cc A))$ equals
$-\wt{\otimes}\bb G_m^{\cc A}$. $DM_{\cc A}(k)$ is equivalent to the homotopy category of
the symmetric $\bb G_m^{\cc A}$-spectra associated to a monoidal motivic model
category structure on $\Ch(Sh(\cc A))$. We also refer the reader to~\cite{DF},
where a monoidal model structure is defined in the case of $MW$-correspondences.

}\end{rem}

\section{The additive motivic sphere spectrum $\bb S^{\cc A}$}

Let $Sp_{S^1,\bb G_m}(k)$ denote the category of symmetric $(S^1,\bb
G_m)$-bispectra, where the $\bb G_m$-direction is associated with
the pointed motivic space $(\bb G_m,1)$. It is equipped with a
stable motivic model category structure~\cite{Jar2}. Denote by $SH(k)$
its homotopy category. The category $SH(k)$ has a closed symmetric
monoidal structure with monoidal unit being the motivic sphere
spectrum $\bb S$ (see~\cite{Jar2} for details). Given $p>0$, the
category $Sp_{S^1,\bb G_m}(k)$ has a further model structure whose
weak equivalences are the maps of bispectra $f:X\to Y$ such that
the induced map of bigraded Nisnevich sheaves
$f_*:\underline{\pi}_{*,*}^{\bb A^1}(X)\otimes\bb Z[\frac1p]\to\underline{\pi}_{*,*}^{\bb A^1}(Y)\otimes\bb Z[\frac1p]$
is an isomorphism. In what follows we denote its homotopy category by
$SH(k)[p^{-1}]$. The category $SH(k)_{\bb Q}$ is defined in a similar fashion.
The corresponding classes of weak equivalences are also called
{\it $p^{-1}$-stable/$\bb Q$-stable motivic weak equivalences}.
We also refer the reader to~\cite[Appendix~A]{RO1} for
general localization theory of motivic spectra.\label{bobrovsky}

It is worth to mention that any other kind of motivic spectra or
motivic functors in the sense of~\cite{DRO1} together with the
stable motivic model structure lead to equivalent definitions of
$SH(k)[p^{-1}]$ and $SH(k)_{\bb Q}$ respectively.

The isomorphism $\tw:(\bb G_m,1)\wedge(\bb G_m,1)\lra{\cong}(\bb
G_m,1)\wedge(\bb G_m,1)$ permuting factors is an involution, i.e.
$\tw^2=\id$. It gives an endomorphism $\epsilon:\bb S\to\bb S$ such
that $\epsilon^2=\id$. If we denote by $SH(k)[2^{-1}]$ the stable
motivic homotopy theory with $\bb
Z\left[\frac12\right]$-coefficients, then
   $$\epsilon_+=-\frac{\epsilon -1}2\quad{\textrm{and}}\quad\epsilon_-=\frac{\epsilon+1}2$$
are two orthogonal idempotent endomorphisms of $\bb S[2^{-1}]$ such
that $\epsilon_++\epsilon_-=\id$ and
$\epsilon=\epsilon_--\epsilon_+$. It follows that
   $$\bb S[2^{-1}]=\bb S_+\oplus\bb S_-,$$
where $\bb S_+$ (respectively $\bb S_-$) corresponds to the
idempotent $\epsilon_+$ (respectively $\epsilon_-$).

By~\cite[Section~6]{MorICTP} the stable algebraic Hopf map $\eta:\bb
S\to\bb S^{-1,-1}$ satisfies $\eta\epsilon_+=0$, $\eta\epsilon=\eta$
in $SH(k)[2^{-1}]$. Moreover,
   $$\bb S_-\hookrightarrow\bb S[2^{-1}]\xrightarrow\eta\bb S^{-1,-1}[2^{-1}]\twoheadrightarrow\bb S^{-1,-1}_-$$
is an isomorphism in $SH(k)[2^{-1}]$, denoted by the same letter
$\eta$. In particular, there is an isomorphism
   $$\bb S_-\cong\bb S[\eta^{-1},2^{-1}]=\hocolim_{SH(k)[2^{-1}]}
     (\bb S\lra{\eta}\bb S^{-1,-1}\lra{\eta}\bb S^{-2,-2}\lra{\eta}\cdots).$$

The decomposition $\bb S[2^{-1}]=\bb S_+\oplus\bb S_-$ of the
monoidal unit of $SH(k)[2^{-1}]$ implies $SH(k)[2^{-1}]$ is a
product of symmetric monoidal triangulated categories
   $$SH(k)[2^{-1}]=SH(k)_+\times SH(k)_-,$$
where $\bb S_+$ and $\bb S_-$ are monoidal units for $SH(k)_+$ and
$SH(k)_-$ respectively.

Consider a category of correspondences $\cc A$. There is
a natural triangulated functor
   $$F:SH(k)\to D_{\bb A^1}^{st}(Sh(\cc A)).$$
In more detail, there is an adjoint pair~\cite[Section~6]{GP}
   $$SH_{S^1}(k)\rightleftarrows\Ho(\Mod \cc A^{EM}),$$
where $\Mod \cc A^{EM}$ is the category of $\cc A^{EM}$-modules equipped with the stable projective motivic
model structure over the Eilenberg--Mac~Lane spectral category $\cc A^{EM}$ associated with $\cc A$. Also,
there is a zig-zag of triangulated equivalences between $\Ho(\Mod\cc A^{EM})$ and $D_{\bb A^1}(Sh(\cc A))$.
Then the resulting functor
   $$SH_{S^1}(k)\to D_{\bb A^1}(Sh(\cc A))$$
is naturally extended to $\bb G_m$-spectra in both categories.

The functor $F$ sends each bispectrum $\Sigma^\infty_{S^1}\Sigma^\infty_{\bb
G_m}X_+$, $X\in Sm/k$, to a $\bb G_m$-spectrum isomorphic to
   $$\cc A(X)_{\bb G_m}^\infty:=(\cc A(-,X)_{\nis},\cc A(-,X\wedge\bb G_m^{\wedge 1})_{\nis},
     \cc A(-,X\wedge\bb G_m^{\wedge2})_{\nis},\ldots)\in D_{\bb A^1}^{st}(Sh(\cc A)).$$
Here each entry is a complex in degree zero, each $\cc
A(-,X\wedge\bb G_m^{\wedge n})_{\nis}$ is a sheaf associated to the
presheaf
   $$\cc A(-,X\times\bb G_m^{\wedge n}))=\cc A(-,X\times\bb G_m^{\times n}))/
     \sum^n_{s=1}(i_s)_*\cc A(-,X\times\bb G_m^{\times n-1})),$$
where the natural additive functors $i_s:\cc A(-,X\times\bb
G_m^{\times(\ell-1)})\to\cc A(-,X\times\bb G_m^{\times\ell})$ are
induced by the embeddings $i_s:\bb G_m^{\times(\ell-1)}\to\bb
G_m^{\times\ell}$ of the form
   $$(x_1,\ldots,x_{\ell-1})\longmapsto(x_1,\ldots,1,\ldots,x_{\ell-1}),$$
where 1 is the $s$th coordinate.

Note that $F$ factors through the stable $\bb A^1$-derived category $D_{\bb A^1}(k):=D_{\bb A^1}^{st}(Sh(\cc A_{naive}))$
in the sense of Morel~\cite{Mor0} (see~\cite[Section~5.3]{CD} as well). In what follows we shall denote by
$H_{\bb A^1}\bb Z$ its monoidal unit. Note that $H_{\bb A^1}\bb Z$ is the image of $\bb S$ under the
canonical functor
   $$SH(k)\to D_{\bb A^1}(k).$$

As above, one has decompositions
   $$H_{\bb A^1}\bb Z[2^{-1}]=H_{\bb A^1}\bb Z_{+}\oplus H_{\bb A^1}\bb Z_{-},\quad  D_{\bb A^1}(k)[2^{-1}]= D_{\bb A^1}(k)_+\times D_{\bb A^1}(k)_-.$$

In what follows we shall write $\bb S^{\cc A}$ to denote the
spectrum $\cc A(pt)_{\bb G_m}^\infty$ and call it the {\it additive motivic
$\cc A$-sphere spectrum}. Taking the Eilenberg--Mac~Lane $S^1$-spectra for
each sheaf $\cc A(-,X\wedge\bb G_m^{\wedge n})_{\nis}$ (see, e.g.,~\cite[\S3.2]{Mor})
we can regard $\bb S^{\cc A}$ as an ordinary $(S^1,\bb G_m)$-bispectrum (and denote it
by the same letter if there is no likelihood of confusion).

The canonical triangulated functor
   $$F:SH(k)\to D_{\bb A^1}^{st}(Sh(\cc A))$$
takes the ordinary motivic sphere $\bb S$ to a spectrum isomorphic to $\bb S^{\cc A}$. $F(\eta)$ induces a morphism
   $$\eta_{\cc A}:\bb S^{\cc A}\to(\bb S^{\cc A})^{-1,-1}.$$
We also set
   $$\bb S^{\cc A}[\eta_{\cc A}^{-1}]:=\hocolim_{D_{\bb A^1}^{st}(Sh(\cc A))}(\bb S^{\cc A}\xrightarrow{\eta_{\cc A}}(\bb S^{\cc A})^{-1,-1}
     \xrightarrow{\eta_{\cc A}}(\bb S^{\cc A})^{-2,-2}\xrightarrow{\eta_{\cc A}}\cdots)$$
and $\bb S^{\cc A}_-\cong F(\bb S_-)$, $\bb S^{\cc A}_+\cong F(\bb S_+)$. Then we have the following relations in $D_{\bb A^1}^{st}(Sh(\cc A))$:
   $$\bb S^{\cc A}[2^{-1}]=\bb S^{\cc A}_+\oplus\bb S^{\cc A}_-\quad{\textrm{and}}\quad\bb S^{\cc A}_-\cong\bb S^{\cc A}[\eta_{\cc A}^{-1},2^{-1}].$$
As above, $\eta_{\cc A}$ annihilates $\bb S^{\cc A}_+$ and is an isomorphism on $\bb S^{\cc A}_-$.

\begin{rem}{\rm
Following an equivalent description of $DM_{\cc A}(k)$ over a symmetric monoidal strict $V$-category of correspondences
in Remark~\ref{dmsymm} in terms of $\bb G_m^{\cc A}$-symmetric spectra (in this case $D_{\bb A^1}^{st}(Sh(\cc A)),DM_{\cc A}(k)$ are
canonically equivalent), the additive motivic $\cc A$-sphere spectrum $\bb S^{\cc A}$ is nothing but the symmetric sequence
   $$(\cc A(-,pt),\bb G_m^{\cc A},\bb G_m^{\cc A}\wt{\otimes}\bb G_m^{\cc A},\ldots,(\bb G_m^{\cc A})^{\wt{\otimes}n},\ldots),$$
where $\Sigma_n$ acts on $(\bb G_m^{\cc A})^{\wt{\otimes}n}$ by permutation. It is a commutative monoid
in the category of symmetric sequences in $\Ch(Sh(\cc A))$ (see~\cite[Section~7]{Hov}). Moreover,
the motivic model category $Sp^\Sigma(\Ch(Sh\cc A),\bb G_m^{\cc A})$ of symmetric $\bb G_m^{\cc A}$-spectra
associated with the motivic model category structure on $\Ch(Sh\cc A)$ is the category of modules in the category
of symmetric sequences over the commutative monoid $\bb S^{\cc A}$. The homotopy category of
$Sp^\Sigma(\Ch(Sh\cc A),\bb G_m^{\cc A})$, which is equivalent to $DM_{\cc A}(k)$,
is a closed symmetric monoidal category with $\bb S^{\cc A}$ a monoidal unit.

}\end{rem}

\begin{defs}{\rm
Let $\cc A$ be a category of $V$-correspondences.
Following~\cite{Voe1,SV1,Sus} the {\it $\cc A$-motive of a smooth
algebraic variety\/} $X\in Sm/k$, denoted by $M_{\cc A}(X)$, is the complex associated to the
simplicial Nisnevich sheaf
   $$n\longmapsto\cc A(-\times\Delta^n,X)_{\nis},\quad\Delta^n=\spec k[t_0,\ldots,t_n]/(t_0+\cdots+t_n-1).$$

}\end{defs}

\begin{lem}\label{aspectra}
Let $\cc A$ be a strict category of $V$-correspondences and $\cc X$
a motivic $S^1$-spectrum such that its presheaves $\pi_*(\cc X)$ of
homotopy groups are homotopy invariant $\cc A$-presheaves. Then
every Nisnevich local fibrant replacement $\cc X_f$ of $\cc X$ is
motivically fibrant.
\end{lem}

\begin{proof}
Since $\cc A$ is a strict category of $V$-correspondences, the
sheaves $\pi_*(\cc X)_{\nis}$ are strictly $\bb A^1$-invariant. Our
claim now follows from~\cite[6.2.7]{Mor}.
\end{proof}

\begin{rem}{\rm
It is worth to mention that Lemma~\ref{aspectra} does not depend on
Morel's connectivity theorem~\cite[6.1.8]{Mor}. Indeed, it easily
follows for connected spectra from Brown--Gersten spectral sequence.
Then we use the fact that $\cc X_f=\hocolim_{n\to-\infty} (\cc
X_{\geq n})_f$, where $\cc X_{\geq n}$ is the naive $n$th truncation
of $\cc X$.

}\end{rem}

The spectrum $\cc A(X)_{\bb G_m}^\infty$ is motivically equivalent
to\label{formone}
   \begin{equation*}\label{bispa}
    M_{\cc A}^{\bb G_m}(X):=(M_{\cc A}(X),M_{\cc A}(X\wedge\bb G_m^{\wedge 1}),
        M_{\cc A}(X\wedge\bb G_m^{\wedge2}),\ldots).
   \end{equation*}
of Nisnevich $\cc A$-sheaves associated with the simplicial sheaf
$n\longmapsto\cc A(-\times\Delta^n,X\wedge\bb G_m^{\wedge n})_{\nis}$.

\begin{defs}\label{kogom}{\rm
Let $\cc A$ be a category of $V$-correspondences. The {\it bivariant
$\cc A$-motivic cohomology groups\/} are defined by
    $$H_{\cc A}^{p,q}(X,Y):=H^p_{\nis}(X,\cc A(-\times\Delta^\bullet,Y\wedge\bb G_m^{\wedge q})_{\nis}[-q]),$$
where the right hand side stands for Nisnevich hypercohomology
groups of $X$ with coefficients in $\cc
A(-\times\Delta^\bullet,Y\wedge\bb G_m^{\wedge q})_{\nis}[-q]$ (the
shift is cohomological).

Following~\cite{GP2} we say that the bigraded presheaves $H^{*,*}_{\cc A}(-,Y)$ satisfy
the {\it cancellation property\/} if all maps
   $$\beta^{p,q}:H^{p,q}_{\cc A}(X,Y)\to H^{p+1,q+1}_{\cc A}(X\wedge\bb G_m,Y).$$
induced by the structure maps of the spectrum $M_{\cc A}^{\bb G_m}(Y)$ are
isomorphisms.

}\end{defs}

Given $Y\in Sm/k$, denote by
   $$M_{\cc A}^{\bb G_m}(Y)_f:=(M_{\cc A}(Y)_f,M_{\cc A}(Y\wedge\bb G_m^{\wedge 1})_f,M_{\cc A}(Y\wedge\bb G_m^{\wedge2})_f,\ldots),$$
where each $M_{\cc A}(Y\wedge\bb G_m^{\wedge n})_f$ is a fibrant
Nisnevich local replacement of $M_{\cc A}(Y\wedge\bb G_m^{\wedge
n})$. It is important to note that each $M_{\cc A}(Y\wedge\bb
G_m^{\wedge n})_f$ can be constructed within $\Ch(Sh(\cc A))$
whenever $\cc A$ is a strict category of $V$-correspondences. (this
can be shown similar to~\cite[5.12]{GP}). Observe as well that $\cc
A(Y)_{\bb G_m}^\infty$ is motivically equivalent to $M_{\cc A}^{\bb
G_m}(Y)_f$.

\begin{lem}\label{brus}
Suppose $\cc A$ is a strict $V$-category of correspondences. The
bigraded presheaves $H^{*,*}_{\cc A}(-,Y)$ satisfy the cancelation
property if and only if $M_{\cc A}^{\bb G_m}(Y)_f$ is motivically
fibrant as an ordinary motivic bispectrum.
\end{lem}

\begin{proof}
Using Lemma~\ref{aspectra}, this is proved similar to~\cite[4.5]{GP2}.
\end{proof}

\begin{cor}\label{bruscor}
Suppose $\cc A$ is a strict $V$-category of correspondences satisfying the
cancellation property. Then the presheaves $H^{*,*}_{\cc A}(-,Y)$ are represented in $SH(k)$ by the
bispectrum $M_{\cc A}^{\bb G_m}(Y)_f$. Precisely,
   $$H^{p,q}_{\cc A}(X,Y)=SH(k)(X_+,S^{p,q}\wedge M_{\cc A}^{\bb G_m}(Y)_f),\quad p,q\in\bb Z,$$
where $S^{p,q}=S^{p-q}\wedge(\bb G_m,1)^{\wedge q}$.
\end{cor}

Under the assumptions of Corollary~\ref{bruscor} we can compute $\bb S^{\cc A}[\eta_{\cc A}^{-1}]$
up to an isomorphism in $D^{st}_{\bb A^1}(Sh(\cc A))$ as follows.
   $$\bb S^{\cc A}[\eta_{\cc A}^{-1}]\cong\hocolim_{DM_{\cc A}(k)}(M_{\cc A}^{\bb G_m}(pt)_f\xrightarrow{}\Omega_{(\bb G_m,1)}(M_{\cc A}^{\bb G_m}(pt)_f)
     \xrightarrow{}\Omega_{(\bb G_m,1)^{\wedge 2}}(M_{\cc A}^{\bb G_m}(pt)_f)\xrightarrow{}\cdots).$$
Here the maps of the colimit are induced by $\eta_{\cc A}$. Denote
the right hand side by $M_{\cc A}^{\bb G_m}(pt)_f[\eta^{-1}]$. It is
termwise a spectrum
   $$M_{\cc A}^{\bb G_m}(pt)_f[\eta^{-1}]=(\Omega_{(\bb G_m,1)}^\infty(M_{\cc A}(pt)_f),\Omega_{(\bb G_m,1)}^{\infty-1}(M_{\cc A}(pt)_f),\Omega_{(\bb G_m,1)}^{\infty-2}(M_{\cc A}(pt)_f),\ldots),$$
where each
   $$\Omega_{(\bb G_m,1)}^{\infty-n}(M_{\cc A}(pt)_f):=\hocolim_{DM_{\cc A}^{eff}(k)}(M_{\cc A}(\bb G_m^{\wedge n})_f\xrightarrow{}\Omega_{(\bb G_m,1)}(M_{\cc A}(\bb G_m^{\wedge n})_f)\to\cdots).$$

Since the structure maps of $M_{\cc A}^{\bb G_m}(pt)_f[\eta^{-1}]$
are schemewise equivalences by the cancellation property, it follows
from the construction of $M_{\cc A}^{\bb G_m}(pt)_f[\eta^{-1}]$ that
all homotopy sheaves $\pi_{i,j}^{\bb A^1}(M_{\cc A}^{\bb
G_m}(pt)_f[\eta^{-1}])$ are concentrated in weight zero only.
By~\cite[4.3.11]{MorICTP} the canonical map of sheaves
   $$\pi_n^{\bb A^1}(\Omega_{(\bb G_m,1)}(\Omega_{(\bb G_m,1)}^{\infty-n}(M_{\cc A}(pt)_f))\to\pi_n^{\bb A^1}(\Omega_{(\bb G_m,1)}^{\infty-n}(M_{\cc A}(pt)_f)_{-1},\quad n\in\bb Z,$$
is an isomorphism, hence the composite map of sheaves is an
isomorphism for all $n\geq 0$
   \begin{gather*}
     \beta_n:\pi_{-n,-n}^{\bb A^1}(M_{\cc A}^{\bb G_m}(pt)_f[\eta^{-1}])\to\pi_{-n-1,-n-1}^{\bb A^1}(\Omega_{(\bb G_m,1)}M_{\cc A}^{\bb G_m}(pt)_f[\eta^{-1}])\to\\
     \to(\pi_{-n-1,-n-1}^{\bb A^1}(M_{\cc A}^{\bb G_m}(pt)_f[\eta^{-1}]))_{-1},
   \end{gather*}
where the left map is induced by the structure map.

Denote by $\cc W^{\cc A}$ the strictly $\bb A^1$-invariant sheaf $\pi_{0,0}^{\bb A^1}(M_{\cc A}^{\bb G_m}(pt)_f[\eta^{-1}])$. If
we regard it as a complex concentrated in zeroth degree, then the collection of complexes
   $$\cc W^{\cc A}_{\bb G_m}:=(\cc W^{\cc A},\cc W^{\cc A},\cc W^{\cc A},\ldots)$$
together with isomorphisms $\beta_0:\cc W^{\cc A}\to({\cc W}^{\cc A})_{-1}$ is an object of $DM_{\cc A}(k)$,
which is $\bb A^1$-local as an ordinary bispectrum (after taking the Eilenberg--Mac~Lane spectrum of each sheaf $\cc W^{\cc A}$).
Notice that the homotopy module of $\cc W^{\cc A}_{\bb G_m}$ in the sense of~\cite[5.2.4]{Mor0} is given by $(M_*,\mu_*)$
with each $M_n=\cc W^{\cc A}$ and $\mu_n=\beta_0$, $n\in\bb Z$. There is a canonical morphism of spectra
   \begin{equation}\label{spor}
    H:M_{\cc A}^{\bb G_m}(pt)_f[\eta^{-1}]\to\cc W^{\cc A}_{\bb G_m},
   \end{equation}
induced by taking the zeroth homology sheaf of each complex $\Omega_{(\bb G_m,1)}^{\infty-n}(M_{\cc A}(pt)_f)$.

Let $\underline{W}$ be the Nisnevich sheaf of Witt rings on $Sm/k$. Following~\cite[p.~380]{ALP}
we take the model $\underline{W}:=\underline{K}_0^{MW}/h$.
The isomorphism $\underline{W}\cong\pi_{n>0,n>0}^{\bb A^1}(\bb S)$ gives the canonical isomorphism of sheaves
$\epsilon:\underline{W}\cong\underline{\Hom}((\bb G_m,1),\underline{W})$. More precisely, it takes $w\in\underline{W}(U)$ to
$p_1^*(\eta\cdot[t])\cdot p_2^*(w)\in\underline{W}(U\wedge(\bb G_m,1))$, where $t$ is the canonical unit on
$\bb G_m$ and $[t]\in\underline{K}_1^{MW}(\bb G_m)$ the corresponding section.

\begin{defs}\label{pervy}{\rm
Suppose $\cc A$ is a strict $V$-category of correspondences satisfying the
cancellation property and $R$ a flat $\bb Z$-algebra. We say that the spectrum $\bb S^{\cc A}[\eta_{\cc A}^{-1}]$
is {\it of Witt type with $R$-coefficients\/} if the zeroth cohomology sheaf $\cc W^{\cc A}_{R}=\pi_{0,0}^{\bb A^1}(M_{\cc A}^{\bb G_m}(pt)_f[\eta^{-1}])\otimes R$
of the complex $\Omega_{(\bb G_m,1)}^{\infty}(M_{\cc A}(pt)_f)\otimes R$ is the only non-zero cohomology sheaf
(the other cohomology sheaves are required to be zero) and $\cc W^{\cc A}_{R}$ is isomorphic to the Nisnevich sheaf $\underline{W}_{R}=\underline{W}\otimes{R}$.
We also require the diagram
   $$\xymatrix{\cc W^{\cc A}_{R}\ar[r]^(.4){\beta_0}\ar[d]_{\cong}&(\cc W^{\cc A}_{R})_{-1}\ar[d]^{\cong}\\
               \underline{W}_{R}\ar[r]_(.4){\epsilon}&(\underline{W}_{R})_{-1}}$$
to be commutative. If $R=\bb Z$ then we just say that $\bb S^{\cc A}[\eta_{\cc A}^{-1}]$ is {\it of Witt type}.

}\end{defs}

\begin{lem}\label{mkl}
Suppose $\cc A$ is a strict $V$-category of correspondences satisfying the
cancellation property and $R$ a ring of fractions of $\bb Z$. If the spectrum $\bb S^{\cc A}[\eta_{\cc A}^{-1}]$
is of Witt type with $R$-coefficients then it is isomorphic in $SH(k)$ to the bispectrum
   $$\underline{W}_R^{\bb G_m}:=(\underline{W}_R,\underline{W}_R,\ldots),$$
in which every structure map is induced by $\epsilon$.
\end{lem}

\begin{proof}
This immediately follows from Definition~\ref{pervy} and the observation that the morphism of spectra~\eqref{spor}
is a motivic equivalence.
\end{proof}

We are now in a position to prove the main result of the section.

\begin{thm}\label{ratsph}
Suppose $\cc A$ is a strict $V$-category of correspondences satisfying the
cancellation property.

$(1)$ If the spectrum $\bb S^{\cc A}[\eta_{\cc A}^{-1}]$
is of Witt type with $\bb Q$-coefficients then the canonical morphism
   $$\bb S_{-}\otimes\bb Q\to\bb S^{\cc A}_{-}\otimes\bb Q$$
is an isomorphism in $SH(k)$.

$(2)$ If the spectrum $\bb S^{\cc A}[\eta_{\cc A}^{-1}]$
is of Witt type with $\bb Z\left[\frac 12\right]$-coefficients then the canonical morphism
   $$H_{\bb A^1}\bb Z_{-}\to\bb S^{\cc A}_{-}$$
is an isomorphism in $SH(k)$.
\end{thm}

\begin{proof}
(1). It follows from~\cite{ALP} that the composite morphism
   $$\bb S_{-}\otimes\bb Q\to\bb S^{\cc A}_{-}\otimes\bb Q\to\underline{W}_{\bb Q}^{\bb G_m}$$
is an isomorphism in $SH(k)$. By Lemma~\ref{mkl} the right morphism is an isomorphism in
$SH(k)$, and hence so is the left one.

(2). It follows from~\cite[Proposition~37]{Bach} that the composite
morphism
   $$H_{\bb A^1}\bb Z_{-}\to\bb S^{\cc A}_{-}\to\underline{W}^{\bb G_m}[2^{-1}]$$
is an isomorphism in $SH(k)$. By Lemma~\ref{mkl} the right morphism is an isomorphism in
$SH(k)$, and hence so is the left one.
\end{proof}

\section{The Milnor--Witt sphere spectrum $\bb S^{MW}$}

Throughout this section $k$ is an infinite perfect field with $\chr
k\not= 2$. We refer the reader to~\cite{CF} for basic facts and
definitions on the category of finite Milnor--Witt correspondences
$\corrt$. It is a strict $V$-category of correspondences
by~\cite{DF}. It follows from~\cite{FO} that $\corrt$ has
cancellation property. We denote the additive sphere spectrum
associated with $\corrt$ by $\bb S^{MW}$.

By~\cite[5.11]{CF} $\corrt(-,Y)$ is a Zariski sheaf, but not a
Nisnevich sheaf in general~\cite[5.12]{CF}. However, $\corrt(-,pt)$
is the Nisnevich sheaf $\underline{K}_0^{MW}$~\cite[4.5]{CF}, which
is homotopy invariant by~\cite[11.3.3]{F}. Since $\corrt$ is a
strict additive $V$-category of correspondences by~\cite{DF}, we see
that the Nisnevich sheaf $\corrt(-,pt)$ is strictly homotopy
invariant. In particular, the normalised complex $M_{MW}(pt)$
associated to the simplicial sheaf
$\corrt(-\times\Delta^\bullet,pt)$ has only one non-trivial homology
sheaf $\underline{K}_0^{MW}$. It follows from~\cite[5.34]{CF} that
$\pi_0^{\bb A^1}(\Hom((\bb G_m,1)^{\wedge n},M_{MW}(pt)))$ is
isomorphic to the sheaf $\underline{W}$ of Witt rings. Thus the
spectrum $\bb S^{MW}[\eta^{-1}]$ is of Witt type.
Theorem~\ref{ratsph} now implies the following

\begin{prop}\label{label}
The canonical morphisms
   $$\bb S_{-}\otimes\bb Q\to\bb S^{MW}_{-}\otimes\bb Q\quad{\textrm{and}}\quad H_{\bb A^1}\bb Z_{-}\to\bb S^{MW}_{-}$$
are isomorphisms in $SH(k)$.
\end{prop}

It follows from properties of finite Milnor--Witt
correspondences~\cite{CF} that $\bb S_+^{MW}$ is isomorphic to $\bb
S^{\corr}[2^{-1}]$ in $DM_{MW}(k)[2^{-1}]:=DM_{\corrt}(k)[2^{-1}]$. Thus we have a splitting
   $$\bb S^{MW}[2^{-1}]\cong\bb S^{\corr}[2^{-1}]\oplus\bb S_-^{MW}.$$

A theorem of Cisinski--D\'eglise~\cite[16.2.13]{CD} shows that the
canonical map $\bb S_+\otimes{\bb Q}\to\bb S^{\corr}\otimes\bb Q$ is
an isomorphism in $SH(k)$. Combining this with
Proposition~\ref{label}, we have proved the main result of the
section:

\begin{thm}\label{ratsphcw}
Given an infinite perfect field $k$ of characteristic not 2, the canonical morphism of bispectra
   $$\bb S\otimes\bb Q\to\bb S^{MW}\otimes\bb Q$$
is an isomorphism in $SH(k)$.
\end{thm}

Let $H_{\bb A^1}^{*,*}(X)$ be the cohomology theory represented in
$SH(k)$ by the bispectrum $H_{\bb A^1}\bb Z$. The following statement is a consequence of
the preceding theorem and a result of D\'eglise--Fasel~\cite[4.2.6]{DF}:

\begin{cor}\label{ratsphcwcor}
Given an infinite perfect field $k$ of characteristic not 2, $n\geq 0$ and $X\in Sm/k$,
there is a natural isomorphism
   $$H_{\bb A^1}^{2n,n}(X)\otimes\bb Q\cong\wt{CH}{}^n(X)\otimes\bb Q,$$
where the right hand side is the $n$-th rational Chow--Witt group of $X$. In 
particular, if $-1$ is a sum of squares in $k$, then 
  $H_{\bb A^1}^{2n,n}(X)\otimes\bb Q\cong{CH}{}^n(X)\otimes\bb Q,$
where ${CH}{}^n(X)\otimes\bb Q$ is the $n$-th rational Chow group of $X$.
\end{cor}

\section{Reconstructing $SH(k)_{\bb Q}$ from finite Milnor--Witt correspondences}

In this section we prove the main result of the paper stating that
$SH(k)_{\bb Q}$ is recovered as $DM_{MW}(k)_{\bb Q}$ whenever the
base field $k$ is infinite perfect of characteristic not 2. To this
end, we need to extend R\"ondigs--{\O}stv{\ae}r's theorem~\cite{RO}
to preadditive categories of correspondences. Throughout this
section $\cc A$ is a category of correspondences.

Following~\cite[Section~2]{RO} define the category $\bb M^{\cc A}$
of motivic spaces with $\cc A$-cor\-respondences as all
contravariant additive functors from $\cc A$ to simplicial abelian
groups. A scheme $U$ in $Sm/k$ defines a representable motivic space
$\cc A(-,U)\in \bb M^{\cc A}$. Let $\cc U:\bb M^{\cc A}\to\bb M$
denote the evident forgetful functor induced by the graph
$Sm/k\to\cc A$. It has a left adjoint $\bb Z^{\cc A}:\bb M\to\bb
M^{\cc A}$ defined as the left Kan extension functor determined by
   $$\bb Z^{\cc A}((U\times\Delta^n)_+)=\cc A(-,U)\otimes\bb Z(\Delta^n).$$
If $X$ is a motivic space, let $X^{\cc A}$ be short for $\bb Z^{\cc A}(X)$.

Similar to~\cite[\S2.1]{RO} we define a projective motivic model category structure on $\bb M^{\cc A}$. This model category
is denoted by $\bb M^{\cc A}_{\mot}$. The projective motivic model category of motivic spaces is denoted by $\bb M_{\mot}$.
We have a Quillen pair
   $$\bb Z^{\cc A}:\bb M_{\mot}\rightleftarrows \bb M^{\cc A}_{\mot}:\cc U.$$
Using Definition~\ref{gret}(1) and Lemma~\ref{datsyuk}, the proof of
the following lemma literally repeats~\cite[Lemma~9]{RO}.

\begin{lem}\label{roenostv}
A map between motivic spaces with $\cc A$-correspondences is a motivic weak equivalence in $\bb M_{\mot}^{\cc A}$
if and only if it is so when considered as a map between ordinary motivic spaces.
\end{lem}

Let $\iota:pt=\spec k\to\bb G_m$ be the embedding $\iota(pt)=1\in\bb G_m$. The mapping cylinder
yields a factorization of the induced map
   $$\spec k_+\hookrightarrow\Cyl(\iota)\lra{\simeq}(\bb G_{m})_+$$
into a projective cofibration and a simplicial homotopy equivalence in
$\bb M$. Let $\bb G$ denote the cofibrant pointed presheaf $\Cyl(\iota)/\spec k_+$ and $T:=S^1\wedge\bb G$.

Following~\cite[\S2.4]{RO} we define a motivic spectrum
   \begin{equation*}\label{hasp}
    H\cc A=(\cc U(pt_+),\cc U(T^{\cc A}),\cc U((T^{\wedge2})^{\cc A}),\ldots).
   \end{equation*}
The structure maps are induced by morphisms
   $$\alpha_T:(T^{\wedge n})^{\cc A}\to\underline{\Hom}_{\bb M^{\cc A}}(T^{\cc A},(T^{\wedge n+1})^{\cc A})$$
(recall that $Sm/k$ acts on $\cc A$).

Given a symmetric monoidal category of correspondences $\cc A$,
a theorem of Day~\cite{Day} implies that $\bb M^{\cc A}$ is a closed symmetric
monoidal category with a tensor product defined as
   $$X\odot Y=\int^{(U,V)\in\cc A\times\cc A}X(U)\otimes Y(V)\otimes\cc A(-,U\times V).$$
As an example, $\cc A(-,U)\odot\cc A(-,V)=\cc A(-,U\times V)$. 
The monoidal unit equals $\cc A(-,pt)$ with $pt=\spec k$. Similar to~\cite[Example~3.4]{DRO1} $H\cc A$
is a commutative motivic symmetric ring spectrum.

Suppose $\cc A$ is a symmetric monoidal category of correspondences. Repeating the proof of~\cite[Lemma~10]{RO}
word for word, the projective motivic model structure on $\bb M^{\cc A}_{\mot}$ is symmetric monoidal.
Following~\cite{Hov,RO} one can define the stable monoidal model
category of symmetric $T$-spectra $\bb{MSS}^{\cc A}$ associated to
$\bb M^{\cc A}_{\mot}$ (with projective model structure). The
homotopy category of $\bb{MSS}^{\cc A}$ is a model for $D^{st}_{\bb
A^1}(Sh(\cc A))(k)$. It is as well a model for $DM_{\cc A}(k)$
whenever $\cc A$ is a strict $V$-category of correspondences (for
this repeat the proof of~\cite[Theorem~11]{RO} literally).

Below we shall need the following theorem proved by Riou in~\cite[Appendix~B]{LYZ}
(see the proof of~\cite[5.8]{HKO} as well).

\begin{thm}[Riou]\label{riou}
Let $k$ be a perfect field. Let $p$ denote the caracteristic exponent of $k$ (i.e.,
$p > 0$ or $p = 1$ if the characteristic of $k$ is zero). Then, for any smooth finite type $k$-scheme
$U$, the suspension spectrum $\Sigma^\infty_T U_+$ is strongly dualisable in $SH(k)[1/p]$.
\end{thm}

We are now in a position to prove the R\"ondigs--{\O}stv{\ae}r
theorem for $\cc A$-correspondences. Notice that in all known
examples a $V$-category of correspondences is strict whenever the
base field $k$ is (infinite) perfect (of characteristic not 2 if
$\cc A=\wt{\textrm{Cor}}$). We also recall the reader that the
category $SH(k)[p^{-1}]$ is defined on page~\pageref{bobrovsky}. It
is the homotopy category of the stable model category of motivic
functors with weak equivalences being $p^{-1}$-stable motivic
equivalences.

\begin{thm}[R\"ondigs--{\O}stv{\ae}r]\label{rondigsostvaer}
If $k$ is a perfect field of exponential characteristic $p$ and $\cc
A$ is a symmetric monoidal category of correspondences, then the
homotopy category of $\Mod H\cc A_{\bb Z[\frac 1p]}$ (respectively
$\Mod H\cc A_{\bb Q}$) is equivalent to $D^{st}_{\bb A^1}(Sh(\cc
A))(k)[\frac 1p]$ (respectively $D^{st}_{\bb A^1}(Sh(\cc
A))(k)\otimes\bb Q$). The equivalence preserves the triangulated
structure. In particular, $\Ho(\Mod H\cc A_{\bb Z[\frac 1p]})$
(respectively $\Ho(\Mod H\cc A_{\bb Q})$) is equivalent to $DM_{\cc
A}(k)[\frac 1p]$ (respectively $DM_{\cc A}(k)\otimes\bb Q$) if $\cc
A$ is a symmetric monoidal strict $V$-category of correspondences.
\end{thm}

\begin{proof}
We verify the theorem for categories with $\bb
Z[\frac1p]$-coefficients, because the proof of the statement for
categories with rational coefficients repeats that for $\bb
Z[\frac1p]$-coefficients {\it word for word}. The proof of the
theorem for categories with $\bb Z[\frac1p]$-coefficients is the
same with the original R\"ondigs--{\O}stv{\ae}r's theorem~\cite{RO}.
The only difference is that we shall have to deal somewhere with
$p^{-1}$-stable weak equivalences of motivic functors instead of
ordinary stable weak equivalences.

We must show that the canonical pair of adjoint (triangulated) functors
   $$\Phi:\Mod H\cc A_{\bb Z[\frac 1p]}\rightleftarrows\bb{MSS}^{\cc A_{\bb Z[\frac 1p]}}:\Psi$$
is a Quillen equivalence ($\Psi$ forgets correspondences).

Similar to~\cite[Lemma~43]{RO} it suffices to prove that the unit of the adjunction
   \begin{equation}\label{adjunit}
    H\cc A_{\bb Z[\frac 1p]}\wedge U_+\to\Psi\Phi(H\cc A_{\bb Z[\frac 1p]}\wedge U_+)
   \end{equation}
is a stable motivic weak equivalence of motivic symmetric spectra for every smooth scheme $U$.
Note that $\Psi\Phi(H\cc A_{\bb Z[\frac 1p]}\wedge U_+)$ is the symmetric spectrum
$(\cc A(-,U)\otimes\bb Z[\frac 1p],(U_+\wedge T)^{\cc A\otimes\bb Z[\frac 1p]},
(U_+\wedge T^{\wedge2})^{\cc A\otimes\bb Z[\frac 1p]},\ldots)$.

By Theorem~\ref{riou}, $U_+$ is dualizable in $SH(k)[p^{-1}]$ for
every $k$-smooth scheme $U$. Suppose $X$ is a motivic functor in the
sense of~\cite{DRO1} and $B$ is a cofibrant finitely presentable
motivic space such that $-\wedge B$ is dualizable in
$SH(k)[p^{-1}]$. When $X$ preserves motivic weak equivalences of
cofibrant finitely presentable motivic spaces, then the evaluation
of the assembly map
   $$X\wedge B\to X\circ(-\wedge B)$$
is a $p^{-1}$-stable weak equivalence between motivic symmetric spectra by~\cite[Corollary~56]{RO}
(though~\cite[Corollary~56]{RO} is proved within an ordinary stable motivic model structure of
motivic functors, it is also true within the $p^{-1}$-stable model structure). We use here
notation and terminology of~\cite{DRO1}. Recall that motivic functors give a model for
motivic symmetric spectra, and hence for $SH(k)$~\cite{DRO1}.

Consider a motivic functor associated with $H\cc A$ (denoted by the same letters)
   $$H\cc A:c\bb M\hookrightarrow\bb M\xrightarrow{\bb Z^{\cc A}}\bb M^{\cc A}\xrightarrow{\cc U}\bb M.$$
Here $c\bb M$ is the full subcategory of $\bb M$ of cofibrant
finitely presentable motivic spaces. $\bb Z^{\cc A}:\bb
M_{\mot}\to\bb M^{\cc A}_{\mot}$ is a left Quillen functor, hence it
preserves motivic weak equivalences between cofibrant motivic
spaces. By Lemma~\ref{roenostv}, $\cc U$ preserves weak equivalences
in $\bb M_{\mot}^{\cc A}$. It follows that $H\cc A$ preserves
motivic equivalences of cofibrant finitely presentable motivic
spaces. Hence,
   \begin{equation*}\label{mnogoo}
    H\cc A\wedge U_+\to H\cc A\circ(-\wedge U_+)
   \end{equation*}
is a $p^{-1}$-stable weak equivalence between motivic symmetric spectra
by~\cite[Corollary~56]{RO}. Similarly,
   \begin{equation}\label{mnogo}
    H\cc A_{\bb Z[\frac 1p]}\wedge U_+\to H\cc A_{\bb Z[\frac 1p]}\circ(-\wedge U_+)
   \end{equation}
is a $p^{-1}$-stable weak equivalence between motivic symmetric spectra.
Obviously,
   $$\underline{\pi}_{*,*}^{\bb A^1}(H\cc A_{\bb Z[\frac 1p]}\circ(-\wedge U_+))\cong
     \underline{\pi}^{\bb A^1}_{*,*}(H\cc A_{\bb Z[\frac 1p]}\circ(-\wedge U_+))\otimes\bb Z[1/p].$$
This is because fibrant replacements of the $T$-spectrum $H\cc
A_{\bb Z[\frac 1p]}\circ(\bb S_k\wedge U_+)$, where $\bb S_k$ is the
motivic sphere spectrum, can be computed in $\bb{MSS}^{\cc A_{\bb
Z[\frac 1p]}}$.

We claim that
   $$\underline{\pi}_{*,*}^{\bb A^1}(H\cc A_{\bb Z[\frac 1p]}\wedge U_+)\cong
       \underline{\pi}^{\bb A^1}_{*,*}(H\cc A_{\bb Z[\frac 1p]}\wedge U_+)\otimes\bb Z[1/p].$$
Indeed, this follows from an isomorphism in $SH(k)$
   \begin{gather*}
    H\cc A_{\bb Z[\frac 1p]}\wedge U_+\cong\hocolim(H\cc A\lra pH\cc A\lra p H\cc A\lra p\cdots)\wedge U_+\\
    \cong\hocolim(H\cc A\wedge U_+\lra pH\cc A\wedge U_+\lra p\cdots).
   \end{gather*}
We see that~\eqref{mnogo} is not only a $p^{-1}$-stable weak equivalence between motivic symmetric spectra,
but also an ordinary stable motivic equivalence.

Ordinary symmetric $T$-spectra
are obtained from motivic spaces by evaluating them at spheres $S^0,T,T^{\wedge 2},\ldots$ (see~\cite[\S3.2]{DRO1}).
The evaluation of the motivic space $H\cc A_{\bb Z[\frac 1p]}\wedge U_+$ is the symmetric $T$-spectrum
   $$(\cc U(pt_+),\cc U(T^{\cc A_{\bb Z[\frac 1p]}}),\cc U((T^{\wedge2})^{\cc A_{\bb Z[\frac 1p]}}),\ldots)\wedge U_+.$$
The evaluation of the motivic space $H\cc A_{\bb Z[\frac 1p]}\circ(-\wedge U_+)$ is the symmetric $T$-spectrum
   $$\Phi\Psi(H\cc A_{\bb Z[\frac 1p]}\wedge U_+)=(\cc U(\cc A_{\bb Z[\frac 1p]}(-,U)),\cc U((U_+\wedge T)^{\cc A_{\bb Z[\frac 1p]}}),         \cc U((U_+\wedge T^{\wedge2})^{\cc A_{\bb Z[\frac 1p]}}),\ldots).$$
Furthermore, the evaluation of the morphism~\eqref{mnogo} is the morphism~\eqref{adjunit}. We see that
the morphism~\eqref{adjunit} is a stable motivic equivalence of motivic symmetric spectra, as was to be shown.
\end{proof}

\begin{rem}{\rm
Very recently Elmanto and Kolderup~\cite{EK} have suggested another approach to the R\"ondigs--{\O}stv{\ae}r theorem
for $\cc A=\wt{\textrm{Cor}}$ that uses Lurie's $\infty$-categorical version of the Barr--Beck theorem.
}\end{rem}

\begin{thm}[Reconstruction]\label{reconstrshq}
If $k$ is an infinite perfect field with $\chr k\not=2$, then $SH(k)_{\bb Q}$ is equivalent to $DM_{MW}(k)_{\bb Q}$. The
equivalence preserves the triangulated structure.
\end{thm}

\begin{proof}
$SH(k)_{\bb Q}$ is equivalent to $D_{\bb A^1}(k)_{\bb Q}$ (see~\cite{Mor0}). By Theorem~\ref{rondigsostvaer}
the latter is equivalent to the homotopy category of $\bb S^{\cc A_{naive}}\otimes\bb Q$-modules. $\bb S^{\cc A_{naive}}\otimes\bb Q$
is motivically equivalent to the commutative monoid spectrum $\bb S\otimes\bb Q$. By~\cite[4.3]{SS1}
the homotopy category of $\bb S^{\cc A_{naive}}\otimes\bb Q$-modules is equivalent to the homotopy
category of $\bb S\otimes\bb Q$-modules.
By Theorem~\ref{ratsphcw} $\bb S\otimes\bb Q$
is motivically equivalent to the commutative monoid spectrum $\bb S^{MW}\otimes\bb Q$. By~\cite[4.3]{SS1}
the homotopy category of $\bb S\otimes\bb Q$-modules is equivalent to the homotopy category of
$\bb S^{MW}\otimes\bb Q$-modules. Wee see that
$SH(k)_{\bb Q}$ is equivalent to the homotopy category of $\bb S^{MW}\otimes\bb Q$-modules.
By Theorem~\ref{rondigsostvaer} the latter category is triangle equivalent to $DM_{MW}(k)_{\bb Q}$, as was to be shown.
\end{proof}

\begin{rem}{\rm
The triangulated equivalence of Theorem~\ref{reconstrshq} is in fact symmetric monoidal.
The main point here is that the natural functor between categories of correspondences
$\cc A_{naive}\to\wt{\textrm{Cor}}$ is extended to a symmetric monoidal triangulated functor 
$D_{\bb A^1}(k)\to DM_{MW}(k)$ (see~\cite[Section~3.3]{DF} as well).
Consider a commutative diagram of natural triangulated functors
   $$\xymatrix{DM_{MW}(k)_{\bb Q}\ar[rr]\ar[d]&&D_{\bb A^1}(k)_{\bb Q}\ar[d]\\
                        \Ho(\Mod-\bb S^{MW}\otimes\bb Q)\ar[rr]&& \Ho(\Mod-\bb S^{naive}\otimes\bb Q)}$$
The proof of the preceding theorem implies that the lower and the vertical functors are equivalences.
We see that the upper functor is an equivalence. It is right adjoint to the functor
$D_{\bb A^1}(k)_{\bb Q}\to DM_{MW}(k)_{\bb Q}$. It follows that the latter functor is an equivalence, too.
It is plainly symmetric monoidal.

}\end{rem}

\section{Comparing motivic complexes with framed and $MW$-correspondences}

In this section we apply the Reconstruction Theorem~\ref{reconstrshq} to compare rational motives with
framed and Milnor--Witt correspondences respectively. Throughout this section the base field $k$ is
infinite perfect of characteristic different from 2.

It is shown in~\cite{GP3} that the suspension bispectrum $\Sigma^{\infty}_{S^1}\Sigma^{\infty}_{\bb G}X_+\in SH(k)$
of a $k$-smooth algebraic variety $X$ is stably equivalent to the bispectrum
      $$M_{fr}^{\bb G}(X)=(M_{fr}(X),M_{fr}(X\times\bb G_m^{\wedge 1}),M_{fr}(X\times\bb G_m^{\wedge 2}),\ldots),$$
each term of which is a twisted framed motive of $X$. Since the functor $\wt{\bb Z}:\cc X\mapsto\wt{\bb Z}[\cc X]$ respects
stable weak equivalences of bispectra, it follows that bispectrum
$\wt{\bb Z}[\Sigma^{\infty}_{S^1}\Sigma^{\infty}_{\bb G}X_+]\in SH(k)$
is stably equivalent to the bispectrum
      $$\bb ZM_{fr}^{\bb G}(X)=(\bb ZM_{fr}(X),\bb ZM_{fr}(X\times\bb G_m^{\wedge 1}),
           \bb ZM_{fr}(X\times\bb G_m^{\wedge 2}),\ldots).$$
By~\cite[1.2]{GPN} the latter bispectrum is stably equivalent to the bispectrum
      $$LM_{fr}^{\bb G}(X):=(LM_{fr}(X),LM_{fr}(X\times\bb G_m^{\wedge 1}),
           LM_{fr}(X\times\bb G_m^{\wedge 2}),\ldots)$$
consisting of twisted linear framed motives in the sense
of~\cite{GP3}. If we take a levelwise Nisnevich local fibrant
replacement of $LM_{fr}(X\times\bb G_m^{\wedge n})$, we get a
bispectrum
   $$LM_{fr}^{\bb G}(X)_f:=(LM_{fr}(X)_f,LM_{fr}(X\times\bb G_m^{\wedge 1})_f,LM_{fr}(X\times\bb G_m^{\wedge 2})_f,\ldots),$$
where each $LM_{fr}(X\times\bb G_m^{\wedge n})_f$ is a Nisnevich
local fibrant replacement of the $S^1$-spectrum $LM_{fr}(X\times\bb
G_m^{\wedge n})$. It follows from the Cancellation Theorem for
linear framed motives~\cite{AGP} that $LM_{fr}^{\bb G}(X)_f$ is a
motivically fibrant bispectrum. In particular, $(\bb S\wedge
X_+)\otimes\bb Q$ is computed locally in the Nisnevich topology as
the bispectrum
   $$LM_{fr}^{\bb G_m}(X)\otimes\bb Q=(LM_{fr}(X)\otimes\bb Q,LM_{fr}(X\times\bb G_m^{\wedge1})\otimes\bb Q,\ldots)$$
consisting of twisted rational linear framed motives of $X$. Each
$S^1$-spectrum $LM_{fr}(X\times\bb G_m^{\wedge n})$ is the
Eilenberg--Mac~Lane spectrum associated with the simplicial
Nisnevich sheaf $\bb ZF(-\times\Delta^\bullet,X\times\bb G_m^{\wedge
n})$ defined in terms of the category of linear framed
correspondences $\bb ZF_*(k)$ and stabilised in the
$\sigma$-direction (see~\cite{GP3} for details).

It is natural to compare twisted complexes defined by various categories of correspondences. There is
constructed a functor in~\cite{DF}
   $$F:Fr_*(k)\to\corrt.$$
It induces morphisms of twisted complexes
   $$f_n:\bb ZF(-\times\Delta^\bullet,X\times\bb G_m^{\wedge n})\to
       \corrt(-\times\Delta^\bullet,X\times\bb G_m^{\wedge n})_{\nis},\quad n\geq 0.$$

A question, originally due to Calm\`es and Fasel, is whether the
$f_n$-s are quasi-isomorphisms of complexes of Nisnevich sheaves.
The following theorem answers this question in the affirmative with
rational coefficients.

\begin{thm}[Comparison]\label{complexesq}
Given an infinite perfect field of characteristic not 2 and a
$k$-smooth scheme $X$, each morphism of complexes of Nisnevich
sheaves
   $$f_n:\bb ZF(-\times\Delta^\bullet,X\times\bb G_m^{\wedge n})\otimes\bb Q\to
        \corrt(-\times\Delta^\bullet,X\times\bb G_m^{\wedge n})_{\nis}\otimes\bb Q,\quad n\geq 0,$$
is a quasi-isomorphism.
\end{thm}

\begin{proof}
We defined the bispectrum $M_{MW}^{\bb G_m}(X)$ on
p.~\pageref{formone}. Taking levelwise Nisnevich local fibrant
replacements, we get a bispectrum $M_{MW}^{\bb G_m}(X)_f$. The
canonical morphism of bispectra
$\Sigma^{\infty}_{S^1}\Sigma^{\infty}_{\bb G}X_+\to M_{MW}(X)_f$
factors as
   $$\Sigma^{\infty}_{S^1}\Sigma^{\infty}_{\bb G}X_+\bl\ell\to LM_{fr}^{\bb G_m}(X)_f\bl F\to M_{MW}^{\bb G_m}(X)_f.$$
As we have shown above, the left arrow is rationally a stable
motivic equivalence. $F$ is a map between fibrant bispectra which
are both locally given by twisted complexes with linear framed and
finite Milnor--Witt correspondences respectively. It follows that
the morphisms of the corollary are quasi-isomorphisms if and only
$(F\circ\ell)\otimes\bb Q$ is an isomorphism in $SH(k)$. But the
latter follows from the Reconstruction Theorem~\ref{reconstrshq}.
\end{proof}

\begin{rem}{\rm
Bachmann and Ananyevskiy pointed out recently to the author that
Theorem~\ref{complexesq} cannot be true with integer coefficients
even for $X=pt$. Moreover, it is not true with $\bb
Z[\frac1{p_1},\ldots,\frac1{p_s}]$-coefficients for any finite
collection of primes $p_1,\ldots,p_s$. Therefore the
quasi-isomorphism of the Comparison Theorem is genuinely rational.
}\end{rem}

\section{Concluding remarks}

The methods developed in the previous sections should also be
applicable to compute $SH(k)_{\bb Q}$ in terms of the hypothetical
category of ``Hermitian correspondences" $K_0^h$. Its objects are
those of $Sm/k$ and morphisms are given by certain bimodules with
duality with/without coefficients in some line bundles. $K_0^h$ is
expected to be a symmetric monoidal strict $V$-category of
correspondences satisfying cancellation property. It is as well
expected that
   $$\bb S^{K_0^h}[2^{-1}]\cong\underline{W}_{\bb Z\left[\frac12\right]}^{\bb G_m}\oplus\bb S^{K_0^{\oplus}}[2^{-1}].$$
Suslin's theorem~\cite{Sus} together with~\cite[3.4]{ALP}
and~\cite[16.2.13]{CD} then would imply that $\bb S\otimes\bb Q$ is
isomorphic to $\bb S^{K_0^h}\otimes\bb Q$. The proof of the
Reconstruction Theorem~\ref{reconstrshq} then would be the same for
showing that $SH(k)_{\bb Q}$ is equivalent to $DM_{K_0^h}(k)_{\bb
Q}$.

The Suslin theorem~\cite{Sus} comparing Grayson's cohomology with
motivic cohomology is then extended to finite Milnor--Witt
correspondences as follows. It states that there is a natural
functor between categories of $V$-correspondences
   $$K_0^h\to\corrt$$
such that the induced morphisms of twisted complexes of Nisnevich sheaves
   $$K_0^h(-\times\Delta^\bullet,\bb G_m^{\wedge n})_{\nis}\to \corrt(-\times\Delta^\bullet,\bb G_m^{\wedge n})_{\nis}$$
is locally a quasi-isomorphism (at least over infinite perfect fields of characteristic not 2).
The extension of the Suslin theorem should be reduced to the original Suslin theorem.

We invite the interested reader to construct the category of ``Hermitian correspondences"
$K_0^h$ with the desired properties.

\end{document}